\documentclass[12pt, reqno]{amsart}
\usepackage{latexsym}
\usepackage{mathtools}
\usepackage{amsmath}
\usepackage{amsthm}
\usepackage{mathrsfs}
\usepackage{lmodern}
\usepackage[T1]{fontenc}
\usepackage{textcomp}
\usepackage[utf8]{inputenc}
\usepackage{color}
\usepackage{amssymb}
\usepackage{enumerate}
\usepackage[abbrev]{amsrefs}
\usepackage{indentfirst}
\usepackage{graphicx}
\usepackage{bmpsize}
\usepackage{hyperref}
\usepackage{enumitem}
\usepackage{enumerate}
\usepackage{url}
\usepackage{autobreak}
\usepackage[justification=justified]{caption}
\usepackage{bm}
\usepackage{here}
\usepackage[labelformat=empty,subrefformat=parens]{subcaption}
\usepackage{tikz-cd}
\renewcommand{\thefootnote}{} 

\theoremstyle{plain} 
\newtheorem{theorem}{\indent\sc Theorem}[section]
\newtheorem{lemma}[theorem]{\indent\sc Lemma}

\newtheorem{proposition}[theorem]{\indent\sc Proposition}

\theoremstyle{definition} 
\newtheorem{definition}[theorem]{\indent\sc Definition}
\newtheorem{remark}[theorem]{\indent\sc Remark}
\newtheorem{example}[theorem]{\indent\sc Example}

\newtheorem{problem}[theorem]{\indent\sc Problem}
\newenvironment{sproof}{%
  \proof}{\endproof}
\makeatletter
  
  \@addtoreset{equation}{section}
\makeatother
\newcommand{\cA}{\mathcal{A}}

\newcommand{\cR}{\mathcal{R}}

\newcommand{\cH}{\mathcal{H}}
\newcommand{\cC}{\mathcal{C}}
\newcommand{\bR}{\mathbb{R}}

\newcommand{\bZ}{\mathbb{Z}}

\newcommand{\stab}[2]{\operatorname{Stab}_{#1}(#2)}
\newcommand{\im}{\operatorname{im}}

\newcommand{\G}[1]{G_0(#1)}
\newcommand{\Gy}[1]{G_y(#1)}
\newcommand{\yG}[1]{{}_yG(#1)}
\newcommand{\yGy}[1]{{}_yG_y(#1)}
\newcommand{\N}{P_n}
\newcommand{\NN}{\mathfrak{C}_n}
\newcommand{\Nn}{\N^{<\omega}}
\newcommand{\x}[2]{{x_{#1}}_{[#2]}}

\newcommand{\y}[2]{y_{#1}^{#2}}

\newcommand{\supp}[1]{\textrm{supp}(#1)}
\DeclareMathOperator{\Homeo}{Homeo}
\newcommand{\LM}{\{\G{n}, \yG{n}, \Gy{n}, \yGy{n}\}}
\newcommand{\ov}[1]{\overline{#1}}
\allowdisplaybreaks[3] 

\newcommand\cinput[2]{\lower#1pt\hbox{\input{#2}}}

\makeatletter\let\@wraptoccontribs\wraptoccontribs\makeatother

\begin{document}

\title{The Lodha--Moore groups and their $n$-adic generalizations are not SCY}
\author{Yuya Kodama and Akihiro Takano}
\contrib[With an appendix by]{Yuya Kodama}
\date{}
\renewcommand{\thefootnote}{\arabic{footnote}}  
\setcounter{footnote}{0} 

\begin{abstract}
  A closed 4-manifold is symplectic Calabi--Yau (SCY) if its canonical class is trivial.
  Friedl and Vidussi proved that Thompson's group $F$ cannot be the fundamental group of any SCY manifold.
  In this paper, we show that its generalizations, called the Brown--Thompson group and the $n$-adic Lodha--Moore groups, cannot be also the fundamental group of any SCY manifold by using their method.
  From this proof, we also show that there exist non-trivial infinitely many examples which satisfy Geoghegan's conjecture.
\end{abstract}

\maketitle

\section{Introduction}
We assume that a 4-manifold $M$ is connected, oriented, and closed.
It is well known that its fundamental group $\pi_1(M)$ is finitely presented.
Conversely, any finitely presented group is isomorphic to the fundamental group of a certain 4-manifold.
Moreover, Gompf \cite{gompf1995symplectic} proved that such a 4-manifold can be constructed to be symplectic.
Thus it is natural to consider a relation between an additional geometric condition of a symplectic manifold with an algebraic condition of its fundamental group.
A symplectic manifold is said to be \textit{symplectic Calabi--Yau} (SCY for short) if its canonical class, the first Chern class of its cotangent bundle, is trivial.
In the 6-dimensional case, SCY manifolds are abundant in the sense that any finitely presented group is realized as the fundamental group of an SCY manifold \cite{fine2013SCY}.
However, the class of SCY 4-manifolds is very restricted.
Indeed, the only known examples of SCY 4-manifolds are the K3 surface and torus bundles over tori, and they are expected to be the only examples \cite{donaldson2008problem, li2006symplectic}.
In particular, Morgan and Szab\'o \cite{morgan1997K3} proved that any simply-connected SCY 4-manifold is homotopy equivalent (therefore homeomorphic by Freedman's result) to the K3 surface.
A finitely presented group $G$ is said to be \textit{SCY} if there exists an SCY 4-manifold $M$ such that $\pi_1(M) \cong G$.
The restriction that a given group is SCY is reflected in several algebraic invariants.
In the case of the first Betti number $b_1(G) = 0$, there are no non-trivial residually finite SCY groups \cite{friedl2013symplectic}.
On the other hand, in the case of $b_1(G) > 0$, the following two invariants are useful to determine whether SCY:
\begin{itemize}
  \item the \textit{virtual first Betti number} $vb_1(G) \coloneqq \sup \{ b_1(H) \mid H\ $is a finite index subgroup of$\ G \}$; and
  \item the \textit{Hausmann--Weinberger invariant} $q(G) \coloneqq \min \{ \chi(M) \mid \ M\ $is a 4-manifold with$\ \pi_1(M) \cong G \}$, where $\chi(M)$ is the Euler characteristic of $M$.
\end{itemize}
A topological interpretation of $vb_1(G)$ is the supremum of the first Betti number of all finite covering spaces of a manifold $M$ with $\pi_1(M) \cong G$.
Note that $vb_1$ is generally larger than $b_1$, and possibly takes value $\infty$.
For example, $vb_1$ is completely calculated for all 3-manifold groups by the virtually Haken conjecture \cite{agol2013haken}.
In particular, if $G$ is a hyperbolic 3-manifold group, then $vb_1(G) = \infty$.
The Hausmann--Weinberger invariant \cite{hausmann1985weinberger} is calculated for some groups such as the fundamental groups of orientable closed surfaces \cite{eckmann1997manifold}, knot groups \cite{eckmann1997manifold}, free abelian groups \cite{hausmann1985weinberger, kirk2005hausmann}, and so on.
Also, Hildum \cite{hildum2015artin} calculated this invariant for some class of right-angled Artin groups and gave a conjecture for all of them.
In general, these two invariants are difficult to determine completely, but if $G$ is an SCY group with $b_1(G) > 0$, then the following are true:
\begin{enumerate}
  \item $2 \leq b_1(G) \leq vb_1(G) \leq 4$ \cite{li2006kodaira, bauer2008complex}; and
  \item if the first $L^2$-Betti number $b_1^{(2)}(G) = 0$, then $q(G) = 0$ \cite{friedl2013symplectic}.
\end{enumerate}
Friedl and Vidussi \cite{friedl2015thompson} proved that Thompson's group $F$ is not SCY by using these facts\footnote{They essentially used only fact (2), since $b_1(F) = vb_1(F) = 2$.}.
They are motivated by the result that Thompson's group $F$ is not K\"ahler \cite{napier2006thompson}.
Also, Bauer \cite{bauer2008complex} asked whether another of Thompson's group $T$ is SCY or not.
This is the case of $b_1 = 0$ and is still an open problem.

Thompson's groups, denoted by $F$, $T$, and $V$, were defined by Richard Thompson in the 1960s.
His first motivation was a connection with logic, and nowadays, they are known to be groups with several mysterious properties.
For instance, it is known that the groups $T$ and $V$ are the first examples of finitely presented infinite simple groups.
The group $F$ is an example of a torsion-free group having type {$\rm F_\infty$} but not type {$\rm F$}.

Many groups similar to Thompson's groups have also been constructed to study such groups, and constructed groups are also interesting.
In this paper, since we will use the invariants already mentioned, we are interested in non-simple groups, or more precisely, non-perfect groups.
Since $T$ and $V$, and most of their generalizations, are perfect, we focus on generalizations of $F$.
The first candidate for such a group is the Brown--Thompson group $F(n)$, introduced by Brown \cite{brown1987finiteness}.
The group $F(n)$ (and $T(n)$) was defined as a subgroup of $V(n)$, where $V(n)$ is a generalization of $V$ defined by Higman \cite{MR0376874}.
The other candidates are the Lodha--Moore group $G_0$ and its generalization.
The group $G_0$ is a group obtained by adding a generator to $F$, defined by Lodha and Moore \cite{lodha2016nonamenable}.
It is the first example of finitely presented, non-amenable, torsion-free with no non-abelian free subgroups.
Its generalization, introduced by the first author \cite{kodama2023n}, is obtained by a natural generalization of $G_0$ and inherits most of the interesting properties of $G_0$.
For example, it also has all the properties mentioned above.

In this paper, for such groups, we show the following:

\begin{theorem}\label{main_1}
  For any $n \geq 2$, the Brown--Thompson group $F(n)$ is not SCY.
\end{theorem}

\begin{theorem} \label{main_2}
  For any $n\geq2$, the $n$-adic Lodha--Moore group $\G{n}$ is not SCY.
\end{theorem}

In a similar way, we also obtain the following result.

\begin{theorem}\label{main_3}
  For every $n\geq2$, the $n$-adic Lodha--Moore group $\G{n}$ has trivial homotopy groups at infinity.
\end{theorem}

It is worth mentioning that this theorem is related to a conjecture about $F$.
In 1979, Geoghegan conjectured that the following hold for $F$:
\begin{enumerate}
  \item it is of type {$\rm F_\infty$};
  \item it has no non-abelian free subgroups;
  \item it is not amenable; and
  \item it has trivial homotopy groups at infinity.
\end{enumerate}
In \cite{brown1984infinite} and \cite{brin1985groups}, (1), (2), and (4) were proved.
However, as is well known, (3) remains open.
Since $G_0$ is similar to $F$, there is a question of whether the four conjectures hold for $G_0$ (and some related groups).
For these groups, Lodha and Moore \cite{lodha2016nonamenable} showed that (2) and (3) hold, Lodha \cite{lodha2020nonamenable} showed that (1) holds, and Zaremsky \cite{zaremsky2016hnn} showed that (4) holds.
In fact, the Lodha--Moore groups are the first examples satisfying all the conjectures.

Since $\G{n}$ is a natural generalization of $G_0$, there is also a question of whether the four conjectures hold for $\G{n}$.
In \cite{kodama2023n}, the first author showed that (2) and (3) also hold for $\G{n}$.
Moreover, it can be shown that (1) naturally follows from a generalization of \cite{lodha2020nonamenable} (Theorem \ref{appendix_main}).
Therefore, from Theorem \ref{main_3}, we conclude that there exist infinitely many groups satisfying Geoghegan's four conjectures.

This paper is organized as follows:
in Section \ref{section_backgrounds}, we first review the definition and properties of the Brown--Thompson group $F(n)$.
Then, we summarize the Lodha--Moore group and its $n$-adic generalization, including some groups not explicitly mentioned so far and their properties.
In Section \ref{section_HNN}, we describe these groups as strictly ascending HNN extensions of certain groups.
These are powerful tools to compute the cohomology groups of these groups with coefficients in their group ring over $\mathbb{Z}$.
In Section \ref{section_notSCY}, we show that $F(n)$ and $\G{n}$ are not SCY by using previous results.
The former result follows immediately from well-known facts, but we describe them for the reader`s convenience.
Finally, in section \ref{section_homotopy}, we give a proof of Theorem \ref{main_3}, using results in Section \ref{section_HNN} and \ref{section_notSCY} again.
An appendix is added in Section \ref{appendix_Finf}.
\section{Backgrounds on generalizations of Thompson's group $F$} \label{section_backgrounds}
For two homeomorphisms $f$ and $g$ on a topological space, we always write $fg$ for $g \circ f$.
All groups introduced in this section are subgroups of homeomorphism groups of some topological spaces.

Let $\N$ be the set $\{0, \dots, n-1\}$.
We endow $\N$ with the discrete topology and endow $\NN \coloneqq \N \times \N \times \cdots$ with the product topology.
This space is homeomorphic to the Cantor set.
Let $\Nn$ be the set of all finite words on $\N$.
We always assume that the empty word $\epsilon$ is in $\Nn$.
For any $s \in \Nn$ and $t \in \Nn$ or $\NN$, their concatenation is written as $st$.
\subsection{Thompson's group $F$ and the Brown--Thompson group $F(n)$}
Fix an integer $n\geq2$.
We define the Brown--Thompson group $F(n)$ as a subgroup of the homeomorphism group of the Cantor set.
For details of Thompson's group $F$ and it generalization $F(n)$, see \cite{cannon1996introductory, burillo2001metrics, brin1998automorphisms}.
%
\begin{definition}
  Let $x_0, x_1, \dots, x_{n-2}$, and $x_{n-1}$ be defined as the following homeomorphisms:
  \begin{align*}
    x_0     & \colon \NN \to\NN;
    \begin{cases}
      0k\eta \mapsto k\eta      & (k<n-1)  \\
      0(n-1)\eta \mapsto (n-1)0\eta        \\
      k\eta \mapsto (n-1)k \eta & (0<k<n),
    \end{cases} \\
    x_1     & \colon \NN \to \NN;
    \begin{cases}
      0\eta \mapsto 0\eta                 \\
      1k \eta \mapsto k\eta    & (k<n-2)  \\
      1(n-2)\eta \mapsto (n-1)0\eta       \\
      1(n-1)\eta \mapsto (n-1)1 \eta      \\
      k\eta \mapsto (n-1)k\eta & (1<k<n),
    \end{cases}  \\
            & \vdots                     \\
    x_{n-2} & \colon\NN \to \NN;
    \begin{cases}
      k \eta \mapsto k \eta          & (k<n-2)       \\
      (n-2)0\eta \mapsto (n-2)\eta                   \\
      (n-2)k \eta \mapsto (n-1)(k-1) & (0<k\leq n-1) \\
      (n-1)\eta \mapsto (n-1)(n-1) \eta,
    \end{cases}
    \shortintertext{and}
    x_{n-1} & \colon \NN \to \NN;
    \begin{cases}
      k\eta \mapsto k \eta & (k<n-1) \\
      (n-1)\eta \mapsto (n-1)x_0(\eta).
    \end{cases}
  \end{align*}
  Then, the \textit{Brown--Thompson group $F(n)$} is a subgroup of $\Homeo(\NN)$ generated by $x_0, x_1, \dots, x_{n-2}$, and $x_{n-1}$. If $n=2$, the group $F(2)$ is called \textit{Thompson's group $F$}.
\end{definition}
For $i\geq n$, let $k_i=\lfloor i/(n-1)\rfloor$, and define $x_i \coloneqq x_0^{-k_i}x_{i-k_i(n-1)}x_0^{k_i}$.
Then, the following group presentations are known \cite{brin1998automorphisms, guba1997diagram}:
\begin{align}
  F(n) & \cong \langle x_0, x_1, x_2, \dots  \mid \text{$x_i^{-1} x_j x_i=x_{j+n-1}$, for $i<j$}  \rangle \label{eq:F(n)-presentation} \\
       & \cong \left\langle x_0, x_1, \dots, x_{n-1} \;\middle|\;
  \begin{array}{l} \text{${x_k}^{x_0}={x_k}^{x_i}$ ($1 \leq i<k\leq n-1$)},  \\
    \text{${x_k}^{x_0x_0}={x_k}^{x_0x_i}$ ($1 \leq i, k\leq n-1$ and $k-1\leq i$)}, \\
    {x_1}^{x_0x_0x_0}={x_1}^{x_0 x_0 x_{n-1}}
  \end{array}
  \right\rangle \notag,
\end{align}
where $x^y$ denotes $y^{-1}x y$.
We will use these presentations in Section \ref{subsection_LM}.

In the proof of Theorem \ref{main_1}, we also use the following well-known fact:
\begin{theorem}[cf.~\cite{brown1987finiteness, brown1984infinite}] \label{BT-property}
  For $n\geq 2$, the following holds:
  \begin{enumerate}
    \item The group $F(n)$ has type {$\rm F_\infty$}, especially, finitely presented.
    \item The group $F(n)$ is one-ended.
    \item The commutator subgroup of $F(n)$ is simple.
    \item The abelianization of $F(n)$ is isomorphic to $\bZ^n$.
    \item The center of the group $F(n)$ is trivial.
  \end{enumerate}
\end{theorem}

\subsection{The Lodha--Moore groups and their $n$-adic generalizations} \label{subsection_LM}
The Lodha--Moore group was first defined as a group consisting of piecewise projective homeomorphisms of the real line and then showed that it is isomorphic to a subgroup of $\Homeo(\mathfrak{C}_2)$ \cite{lodha2016nonamenable}.
However, to the best of our knowledge, its generalization is only defined as a subgroup of $\Homeo(\NN)$.
Hence, in this paper, we also review it following \cite{kodama2023n} and define some new groups, which are generalizations of some groups that appeared in \cite{zaremsky2016hnn}.

Fix $n\geq 2$.
We first define an important element in $\Homeo(\NN)$.
This map of course depends on $n$, but we omit it for the sake of notational simplicity.
\begin{definition}
  The map $y$ and its inverse map $y^{-1}$ is defined recursively as follows:
  \begin{align*}
    y\colon \NN    & \to \NN              & y^{-1}\colon \NN        & \to \NN             \\
    y(00\zeta)     & =0y(\zeta)           & y^{-1}(0\zeta)          & =00y^{-1}(\zeta)    \\
    y(0k\zeta)     & =k\zeta              & y^{-1}(k\zeta)          & =0k\zeta            \\
    y(0(n-1)\zeta) & =(n-1)0y^{-1}(\zeta) & y^{-1}((n-1)0\zeta)     & =0(n-1)y(\zeta)     \\
    y(k\zeta)      & =(n-1)k\zeta         & y^{-1}((n-1)k\zeta)     & =k\zeta             \\
    y((n-1)\zeta)  & =(n-1)(n-1)y(\zeta)  & y^{-1}((n-1)(n-1)\zeta) & =(n-1)y^{-1}(\zeta)
  \end{align*}
  where $k$ is in $\{1, \dots, n-2\}$.
\end{definition}
Using this map, for $\alpha \in \Nn$, define the map $y_\alpha$ by  setting
\begin{align*}
  y_\alpha(\xi) & =
  \left \{
  \begin{array}{cc}
    \alpha y(\eta), & \xi=\alpha\eta    \\
    \xi,            & \mbox{otherwise}.
  \end{array}
  \right.
\end{align*}
\begin{definition}
  We define four groups as follows:
  \begin{align*}
    \G{n}   & \coloneqq \langle x_0, x_1, \dots, x_{n-1}, y_{(n-1)0} \rangle,                   \\
    \yG{n}  & \coloneqq \langle x_0, x_1, \dots, x_{n-1}, y_{(n-1)0}, y_{0} \rangle,            \\
    \Gy{n}  & \coloneqq \langle x_0, x_1, \dots, x_{n-1}, y_{(n-1)0}, y_{(n-1)} \rangle,        \\
    \yGy{n} & \coloneqq \langle x_0, x_1, \dots, x_{n-1}, y_{(n-1)0}, y_{0}, y_{(n-1)} \rangle.
  \end{align*}
  We call these groups the \textit{$n$-adic Lodha--Moore groups}.
  If $n=2$, they are called the \textit{Lodha--Moore groups}.
\end{definition}
In \cite{kodama2023n}, the first author investigated only $\G{n}$.
In the proof of Theorem \ref{main_2}, we use the following facts from among them:
\begin{theorem}[\cite{kodama2023n}] \label{LM-property}
  For $n\geq 2$, the following holds:
  \begin{enumerate}
    \item The group $\G{n}$ is finitely presented.
    \item The group $\G{n}$ is one-ended.
    \item The commutator subgroup of $\G{n}$ is simple. \label{LM-commutator}
    \item The abelianization of $\G{n}$ is isomorphic to $\bZ^{n+1}$.
    \item The center of the group $\G{n}$ is trivial. \label{LM-center}
  \end{enumerate}
\end{theorem}
In the rest of this section, we study mainly the other three groups.
\subsubsection{Presentations of the $n$-adic Lodha--Moore groups}
Following \cite{kodama2023n}, we recall some infinite generating sets and sets of relations for $n$-adic Lodha--Moore groups.
Similar to Thompson's groups, their infinite presentations are also useful to study their algebraic properties.

Let $i \in \{0, \dots, n-2\}$ and $\alpha \in \Nn$.
Define the map $\x{i}{\alpha}\colon \NN \to \NN$ by
\begin{align*}
  \x{i}{\alpha}(\zeta) & =
  \begin{cases}
    \alpha x_i(\eta) & (\zeta=\alpha \eta)       \\
    \zeta            & (\zeta \neq \alpha \eta),
  \end{cases}
\end{align*}
and we define
\begin{align*}
  X(n) & \coloneqq
  \left\{ \x{i}{\alpha} \mid i = 0, \dots, n-2, \alpha \in \Nn \right\}.
\end{align*}
Note that for $k\geq0$, we have $\x{i}{(n-1)^k}=x_{k(n-1)+i}$.
It is well known that each $\x{i}{\alpha}$ is also in $F(n)$, and hence $\x{i}{\alpha}$ is represented by a word on $\{x_0, x_1, \dots \}$.
For any $\x{i}{\alpha}$, choose one such word (for instance, take its normal form) and write it as $w(\x{i}{\alpha})$.
Let $R_{F(n)}$ be the set of relations
\begin{align*}
  \{x_i^{-1}x_jx_{i}=x_{i+(n-1)}, x=w(x) \mid 0\leq i<j, x\in X(n)\}.
\end{align*}
Then the group presented by $\langle X(n) \mid R_{F(n)} \rangle$ is isomorphic to $F(n)$.

Using $y_\alpha$, we define the four sets as follows:
\begin{align*}
  Y_0(n)     & \coloneqq
  \left\{ y_\alpha \;\middle|\;
  \begin{array}{l} \alpha \in \Nn,                                            \\
    \mbox{$\alpha \neq \epsilon, 0^i, (n-1)^i $ for any $i \geq1$} \\
    \mbox{the sum of each number in  $\alpha$ is equal to $0 \bmod {n-1}$}
  \end{array}
  \right\},              \\
  {}_yY(n)   & \coloneqq
  \left\{ y_\alpha \;\middle|\;
  \begin{array}{l} \alpha \in \Nn,                                            \\
    \mbox{$\alpha \neq \epsilon, 0^i, (n-1)^i $ for any $i \geq1$} \\
    \mbox{the sum of each number in  $\alpha$ is equal to $0 \bmod {n-1}$}
  \end{array}
  \right\},              \\
  Y_y(n)     & \coloneqq
  \left\{ y_\alpha \;\middle|\;
  \begin{array}{l} \alpha \in \Nn,                                  \\
    \mbox{$\alpha \neq \epsilon, 0^i$ for any $i \geq1$} \\
    \mbox{the sum of each number in  $\alpha$ is equal to $0 \bmod {n-1}$}
  \end{array}
  \right\},              \\
  {}_yY_y(n) & \coloneqq
  \left\{ y_\alpha \;\middle|\;
  \begin{array}{l} \alpha \in \Nn, \\
    \mbox{the sum of each number in  $\alpha$ is equal to $0 \bmod {n-1}$}
  \end{array}
  \right\},              \\
\end{align*}
Then $n$-adic Lodha--Moore groups $\G{n}$, $\yG{n}$, $\Gy{n}$, and $\yGy{n}$ are also generated by $X(n)\cup Y_0(n)$, $X(n)\cup {}_yY(n)$, $X(n)\cup Y_y(n)$, and $X(n)\cup {}_yY_y(n)$, respectively.

For these generating sets, we claim that the following sets of relations give presentations of $n$-adic Lodha--Moore groups:
\begin{enumerate}
  \item the relations of $F(n)$ in $R_{F(n)}$;
  \item $y_\beta\x{i}{\alpha}=\x{i}{\alpha}y_{\x{i}{\alpha}(\beta)}$ for all $i \in \{0, \dots, n-2\}$ and $\alpha, \beta \in \Nn$ such that $y_\beta \in Y(n)$ and $\x{i}{\alpha}(\beta)$ is defined;
  \item $y_\alpha y_\beta =y_\beta y_\alpha$ for all $\alpha, \beta \in \Nn$ such that $y_\alpha, y_\beta \in Y(n)$, $\alpha$ is not a prefix of $\beta$, and vice versa;
  \item $y_\alpha=\x{0}{\alpha} y_{\alpha0} y_{\alpha(n-1)0}^{-1}y_{\alpha(n-1)(n-1)}$ for all $\alpha \in \Nn$ such that $y_\alpha \in Y(n)$,
\end{enumerate}
where $Y(n)$ is one of $Y_0(n)$, ${}_yY(n)$, $Y_y(n)$, or ${}_yY_y(n)$, and we write $R_0(n)$, ${}_yR(n)$, $R_y(n)$, and ${}_yR_y(n)$ for the collection of these relations, respectively.
Since $\x{i}{\alpha}$ changes only a finite length prefix of each element in $\NN$, it naturally induces a partial action of $\x{i}{\alpha}$ on $\Nn$.
If $\x{i}{\alpha}$ can acts on $\beta$, we say that \textit{$\x{i}{\alpha}(\beta)$ is defined}.
For example, $x_0((n-1))$ is defined, and $x_0(0)$ is not defined.
\begin{theorem}[{cf.~\cite[Corollary 3.38]{kodama2023n}}]
  The following holds:
  \begin{itemize}
    \item the group $\G{n}$ is isomorphic to $\langle X(n)\cup Y_0(n) \mid R_0(n) \rangle$;
    \item the group $\yG{n}$ is isomorphic to $\langle X(n)\cup {}_yY(n) \mid {}_yR(n) \rangle$;
    \item the group $\Gy{n}$ is isomorphic to $\langle X(n)\cup Y_y(n) \mid R_y(n) \rangle$; and
    \item the group $\yGy{n}$ is isomorphic to $\langle X(n)\cup {}_yY_y(n) \mid {}_yR_y(n) \rangle$.
  \end{itemize}
  \begin{sproof}
    The same proof works for all four groups.
    We only give a sketch of proof for $\yGy{n}$.
    See \cite{kodama2023n} for $\G{n}$, and see \cite{lodha2016nonamenable} for $\G{2}$ and $\yGy{2}$.

    By using relations (2) and (4), any word can be rewritten into a word in a concatenation of words on $X(n)$ and ${}_yY_y(n)$.
    In fact, if a word $w$ on $X(n)\cup {}_yY_y(n)$ represents the identity, then by relations (2), (3), and (4), the word can be rewritten into a word on $X(n)$.
    Intuitively, since $w$ represents the identity, it follows that enough use of relation (4) induces cancellations of $y_\alpha$s.
    Note that the correspondence between the recursive definition of the map $y$ and the relation (4).
    Since $\langle X(n) \mid R_{F(n)} \rangle$ is isomorphic to $F$, a word on $X(n)$ representing the identity can be rewritten into the empty word by relations (1).
  \end{sproof}
\end{theorem}

\subsubsection{The abelianizations of the $n$-adic Lodha--Moore groups and commutator subgroups}
In this section, we generalize the results in \cite{burillo2018commutators, zaremsky2016hnn}.
\begin{theorem}[{cf.~\cite[Theorem 4.10]{kodama2023n}, \cite[Corollary 1.3]{zaremsky2016hnn}, and \cite[Lemma 3.1]{burillo2018commutators}}]
  The abelianizations of the $n$-adic Lodha--Moore groups are all isomorphic to $\mathbb{Z}^{n+1}$.
  \begin{proof}
    Let $a\colon F(n) \to \mathbb{Z}^n$ be the surjective homomorphism defined in \cite[section 4D]{brown1987finiteness}.
    Note that $a(x_0)=(1, 0,\dots, 0), a(x_1)=(0, 1, 0, \dots, 0), \dots, a(x_{n-1})=(0, \dots, 0, 1)$ and its kernel is the commutator subgroup $F(n)^\prime$.

    Since we already know that the abelianization of $\G{n}$ is isomorphic to $\bZ^{n+1}$ \cite[Theorem 4.10]{kodama2023n}, we give proofs for the remaining three groups by the following propositions:
    \begin{proposition}\label{Prop_abyG}
      The map $\{x_0, \dots, x_{n-1}, y_{(n-1)0}, y_0\} \to \bZ^{n-1}\oplus \bZ^2$ given by
      \begin{align*}
         & x_0 \mapsto (0, \dots, 0, 1, 0, 0), &  & x_i \mapsto (\bm{e}_i, 0, 0), &  & y_{(n-1)0} \mapsto (\bm{0}, 1, 0), &  & y_0 \mapsto (\bm{0}, 0, 1)
      \end{align*}
      where $i \in \{1, \dots, n-1\}$ and $\bm{e}_i$ is the standard basis, extends to a surjective homomorphism ${}_y\pi\colon\yG{n} \to \bZ^{n+1}$, and its kernel is exactly the commutator subgroup $\yG{n}^\prime$.
      \begin{proof}
        Let $p_1$ be the homeomorphism defined by
        \begin{align*}
          p_1\colon \mathbb{Z}^n \to \mathbb{Z}^{n-1}; (a_1, a_2, \dots, a_{n-1}, a_n)\mapsto (a_2, \dots, a_{n-1}, a_1+a_n).
        \end{align*}
        Define the map $X(n)\cup {}_yY(n) \to \mathbb{Z}^{n+1}$ by setting
        \begin{align*}
           & \x{i}{\alpha} \mapsto (p_1 \circ a(\x{i}{\alpha}), 0, 0), &  & y_{\alpha} \mapsto
          \begin{cases}
            \mbox{$(\bm{0}, 0, 1)$ if $\alpha=0^m$ for some $m \geq 1$}, \\
            \mbox{$(\bm{0}, 1, 0)$ if $\alpha$ is not constant}.
          \end{cases}
        \end{align*}
        Observe that if we restrict this map to be on the set $\{x_0, \dots, x_{n-1}, y_{(n-1)0}, y_0\}$, then it is the same as the map defined in the claim of this proposition.
        Since the map $a$ is surjective, we note that the map $p_1 \circ a \colon F(n) \to \mathbb{Z}^{n-1}$ is also a surjective homomorphism.

        To see that this map extends to ${}_y\pi\colon \yG{n} \to \bZ^{n-1}$, we should compute $p_1 \circ a(\x{0}{\alpha})$.
        From the definition of $a$, if $\alpha$ is not constant, then $a(\x{0}{\alpha})=\bm{0}$, and hence $p_1\circ a(\x{0}{\alpha})=(\bm{0}, 0, 0)$.
        If $\alpha=0^i$ for some $i\geq 1$, then we have $a(\x{0}{0^i})=(1, 0, \dots, 0, -1) \in \bZ^n$ and hence $p_1 \circ a(\x{0}{0^i})=\bm{0}$.
        We also note that the partial action of $F(n)$ on $\Nn$ stabilizes the sets $\{0, 0^2, \cdots\} \subset \Nn$ and $\{\alpha \in \Nn \mid \mbox{$\alpha$ is not constant} \}$.
        From these facts, the map $X(n)\cup {}_yY(n) \to \mathbb{Z}^{n-1}$ extends to a surjective homomorphism ${}_y\pi\colon\yG{n} \to \bZ^{n+1}$.

        We claim that the set $\{x_1\yG{n}^\prime, \dots, x_{n-1}\yG{n}^\prime, y_{(n-1)0}\yG{n}^\prime, y_0\yG{n}^\prime\}$ generates the abelianization ${\yG{n}}/{\yG{n}^\prime}$.
        This follows from the following:
        for $g \in \Gy{n}$, we write $\ov{g}$ for the element $g\yG{n}^\prime \in {\yG{n}}/{\yG{n}^\prime}$.
        Observe that by direct calculations, we have $y_{00}=x_0 y_0 x_0^{-1}$ and
        \begin{align*}
          y_{0(n-1)0}=(x_0^2x_{2(n-1)}x_{n-1}^{-1}x_0^{-1})y_{0(n-1)(n-1)}(x_0^2x_{2(n-1)}x_{n-1}^{-1}x_0^{-1})^{-1}.
        \end{align*}
        These imply that $\ov{y_{00}}=\ov{y_0}$ and $\ov{y_{0(n-1)0}}=\ov{y_{0(n-1)(n-1)}}$ hold.
        Hence, since we have $y_{0}=\x{0}{0}y_{00}y_{0(n-1)0}^{-1}y_{0(n-1)(n-1)}$ by relation (4), the element $\ov{\x{0}{0}}$ is the identity element.
        Again by direct calculation, we have
        \begin{align*}
          \x{0}{0}=x_0^2 x_{n-1}^{-1}x_0^{-1}.
        \end{align*}
        This implies that $\ov{x_0}=\ov{x_{n-1}}$ holds.
        Therefore, the abelianization of $\yG{n}$ is generated by $\{\ov{x_1}, \dots, \ov{x_{n-1}}, \ov{y_{(n-1)0}}, \ov{y_0}\}$.

        Since $\{{}_y\pi(x_1), \dots, {}_y\pi(x_{n-1}), {}_y\pi(y_{(n-1)0}), {}_y\pi(y_0)\}$ is a basis of $\bZ^{n+1}$, ${}_y\pi$ induces an abelianization map ${\yG{n}}/{\yG{n}^\prime} \to \bZ^{n+1}$.
      \end{proof}
    \end{proposition}
    \begin{proposition}\label{Prop_abGy}
      The map $\{x_0, \dots, x_{n-1}, y_{(n-1)0}, y_{n-1}\} \to \bZ^{n-1}\oplus \bZ^2$ given by
      \begin{align*}
         & x_i \mapsto (\bm{e}_{i+1}, 0, 0), &  & x_{n-1} \mapsto (\bm{0}, 0, 0), &  & y_{(n-1)0} \mapsto (\bm{0}, 1, 0), &  & y_{n-1} \mapsto (\bm{0}, 0, 1)
      \end{align*}
      where $i \in \{0, \dots, n-2\}$ and $\bm{e}_i$ is the standard basis, extends to a surjective homomorphism $\pi_y \colon\Gy{n} \to \bZ^{n+1}$, and its kernel is exactly the commutator subgroup $\Gy{n}^\prime$.
      \begin{proof}
        Let $p_{n-1}$ be the homomorphism defined by
        \begin{align*}
          p_{n-1}\colon \mathbb{Z}^n \to \mathbb{Z}^{n-1}; (a_1, \dots, a_{n-1}, a_n)\mapsto (a_1, \dots, a_{n-1}).
        \end{align*}
        Similar to the proof of Proposition \ref{Prop_abyG}, we first consider a map on an infinite generating set of $\Gy{n}$ and then extend it.
        Define the map $X(n)\cup Y_y(n) \to \mathbb{Z}^{n+1}$ by setting
        \begin{align*}
           & \x{i}{\alpha} \mapsto (p_{n-1} \circ a(\x{i}{\alpha}), 0, 0), &  & y_{\alpha} \mapsto
          \begin{cases}
            \mbox{$(\bm{0}, 0, 1)$ if $\alpha=(n-1)^m$ for some $m \geq 1$}, \\
            \mbox{$(\bm{0}, 1, 0)$ if $\alpha$ is not constant}.
          \end{cases}
        \end{align*}
        If we restrict this map to be on the finite generating set, then it is the same as the map defined in the claim of this proposition.
        From the definition of the map $a$, we note that if $\alpha$ is not constant, then $a(\x{0}{\alpha})=\bm{0}$.
        Also, if $\alpha=(n-1)^i$ for some $i\geq1$, then $a(\x{0}{(n-1)^i})=(0, \dots, 0, 1) \in \bZ^n$.
        Hence, the map $X(n)\cup Y_y(n) \to \mathbb{Z}^{n-1}$ extends to the map $\pi_y \colon\Gy{n} \to \bZ^{n+1}$.

        For $g \in \Gy{n}$, we write $\ov{g}$ for the element $g\Gy{n}^\prime \in {\Gy{n}}/{\Gy{n}^\prime}$.
        By direct calculations, we have $y_{(n-1)(n-1)(n-1)}=x_0^{-2}y_{n-1}x_0^2$ and $y_{(n-1)0}=x_0y_{(n-1)(n-1)0}x_0^{-1}$.
        Hence $y_{(n-1)}=\x{0}{(n-1)}y_{(n-1)0}y_{(n-1)(n-1)0}^{-1}y_{(n-1)(n-1)(n-1)}$ induces that $\ov{\x{0}{(n-1)}}$ is the identity element.
        Note that $\x{0}{(n-1)}=x_{n-1}$.
        Since $\{\pi_y(x_0), \dots, \pi_y(x_{n-2}), \pi_y(y_{(n-1)0}), \pi_y(y_{n-1})\}$ is a basis of $\bZ^{n+1}$, the $\pi_y$ induces an abelianization map ${\Gy{n}}/{\Gy{n}^\prime} \to \bZ^{n+1}$.
      \end{proof}
    \end{proposition}
    \begin{proposition}\label{Prop_abyGy}
      The map $\{x_0, \dots, x_{n-1}, y_{(n-1)0}, y_0, y_{n-1}\} \to \bZ^{n-2}\oplus \bZ^3$ given by
      \begin{align*}
         & x_0 \mapsto (\bm{0}, 0, 0, 0)         &  & x_i \mapsto (\bm{e}_{i}, 0, 0, 0), &  & x_{n-1} \mapsto (\bm{0}, 0, 0, 0), & \\
         & y_{(n-1)0} \mapsto (\bm{0}, 1, 0, 0), &  & y_{0} \mapsto (\bm{0}, 0, 1, 0),   &  & y_{n-1} \mapsto (\bm{0}, 0, 0, 1)
      \end{align*}
      where $i \in \{1, \dots, n-2\}$ and $\bm{e}_i$ is the standard basis, extends to a surjective homomorphism ${}_y\pi_y \colon\yGy{n} \to \bZ^{n+1}$, and its kernel is exactly the commutator subgroup $\yGy{n}^\prime$.
      \begin{proof}
        Let $p_{1, n-1}$ be the homomorphism defined by
        \begin{align*}
          p_{1, n-1}\colon \mathbb{Z}^n \to \mathbb{Z}^{n-2}; (a_1, a_2, \dots, a_{n-1}, a_n)\mapsto (a_2, \dots, a_{n-1}).
        \end{align*}
        Similar to the proofs of Propositions \ref{Prop_abyG} and \ref{Prop_abGy}, define the map $X(n)\cup {}_yY_y(n) \to \bZ^{n+1}$ by setting
        \begin{align*}
           & \x{i}{\alpha} \mapsto (p_{1, n-1} \circ a(\x{i}{\alpha}), 0, 0, 0), &  & y_{\alpha} \mapsto
          \begin{cases}
            \mbox{$(\bm{0}, 0, 1, 0)$ if $\alpha=0^m$ for some $m \geq 1$},     \\
            \mbox{$(\bm{0}, 0, 0, 1)$ if $\alpha=(n-1)^m$ for some $m \geq 1$}, \\
            \mbox{$(\bm{0}, 1, 0, 0)$ if $\alpha$ is not constant}.
          \end{cases}
        \end{align*}
        If we restrict this map to be on the finite generating set, then it is the same as the map defined in the claim of this proposition.
        From the proofs of Propositions \ref{Prop_abyG} and \ref{Prop_abGy}, this map extends to the map ${}_y\pi_y \colon \yGy{n} \to \bZ^{n+1}$.
        Since ${}_yR(n)$ and $R_y(n)$ are subsets of ${}_yR_y(n)$, we have that $x_0\yGy{n}^\prime$ and $x_{n-1}\yGy{n}^\prime$ are the identity element in ${\yGy{n}}/{\yGy{n}^\prime}$.
        Since $\{{}_y\pi_y(x_1), \dots, {}_y\pi_y(x_{n-2}), {}_y\pi_y(y_{(n-1)0}), {}_y\pi_y(y_0), {}_y\pi_y(y_{n-1})\}$ is a basis of $\bZ^{n+1}$, the map ${}_y\pi_y$ induces an abelianization map ${\yGy{n}}/{\yGy{n}^\prime} \to \bZ^{n+1}$.
      \end{proof}
    \end{proposition}
    We have the desired result.
  \end{proof}
\end{theorem}
\subsubsection{The center of the $n$-adic Lodha--Moore groups}
Let $D(n)$ be the set $\{s\overline{0} \mid s \in \Nn \}$, where $\overline{0}\coloneqq 00\cdots \in \NN$.
Consider a natural lexicographic order $<$ on $\NN$.
Then the following holds:
\begin{lemma}[{\cite[Lemma 4.16]{kodama2023n}}]\label{lemma_dense}
  For any $s\overline{0}$ in $D(n)$, there exists $x$ in $F(n)$ such that
  \begin{align*}
    \supp{x}=(s\overline{0}, \overline{(n-1)})=\{\xi \in \NN \mid s\overline{0}<\xi< \overline{(n-1)}\}
  \end{align*}
  holds, where $\supp{x}\coloneqq \{\xi \in \NN \mid x(\xi)\neq \xi\}$ and $\ov{(n-1)}=(n-1)(n-1) \cdots \in \NN$.
\end{lemma}
By a similar argument in \cite{kodama2023n}, we obtain the following:
\begin{theorem}[{cf.~\cite[Lemma 4.16]{kodama2023n}, \cite[Theorem 2]{burillo2018commutators}}]
  Let $G$ be in the set $\LM$.
  Then, the center of $G$ is trivial.
  \begin{proof}
    Let $f$ be an element of the center of $G$.
    We first note that if an element $g \in G$ has a support $\supp{g}=\{\xi \in \NN \mid \zeta<\xi< \overline{(n-1)}\}$ for some $\zeta \in \NN$, then $f(\zeta)=\zeta$ holds.
    Indeed, if not, either $f(\zeta)>\zeta$ or $f^{-1}(\zeta)>\zeta$ holds since each element in $G$ preserves the order.
    This implies that $f(\zeta)$ or $f^{-1}(\zeta)$ is in $\supp{g}$.
    Assume that $f(\zeta)$ is in $\supp{g}$.
    Then we have $fg(\zeta)=g(f(\zeta))\neq f(\zeta)$.
    Since $g(\zeta)=\zeta$, we also have $gf(\zeta)=f(g(\zeta))=f(\zeta)$.
    This contradicts the fact that $f$ is in the center of $G$.
    In a similar way, it is also a contradiction if $f^{-1}(\zeta)$ is in $\supp{g}$.

    By Lemma \ref{lemma_dense}, for any $s\overline{0} \in D(n)$, we have $f(s\overline{0})=s\overline{0}$.
    Since $D(n)$ is a dense subset of $\NN$, we obtain the desired result.
  \end{proof}
\end{theorem}

\section{HNN decompositions} \label{section_HNN}
Let $G$ be a group with a presentation $G \cong \left\langle S \mid R \right\rangle$, and $\phi \colon H \to K$ be an isomorphism between two subgroups $H$ and $K$ of $G$.
Let $t$ be a symbol not in $S$.
Then define a group $G*_{\phi,t}$ with the presentation
\begin{align*}
  G*_{\phi,t} \coloneq \left\langle S, t \mid R, t^{-1} h t = \phi(h)\ (h \in H) \right\rangle.
\end{align*}
This group $G*_{\phi,t}$ is called the \textit{HNN extension} of $G$ with respect to $\phi$, and a generator $t$ is called the \textit{stable letter}.
For simplicity, this presentation is denoted by
\begin{align*}
  G*_{\phi,t} \coloneq \left\langle G, t \mid t^{-1} h t = \phi(h)\ (h \in H) \right\rangle.
\end{align*}
If $H = G$, the HNN extension is said to be \textit{ascending}.
Also, if $\im \phi = K$ is a proper subgroup of $G$, then its ascending HNN extension $G*_{\phi,t}$ is called \textit{strict}.

Conversely, suppose that a group $G$ has a structure of an ascending HNN extension of its some subgroup $H$ and a generator $t$.
Then we write simply $G = H*_{t}$ since the homomorphism $\phi$ is just conjugation by $t$. 
Also, we call it an \textit{HNN decomposition} of $G$.
In this section, we show that the Brown--Thompson group and the $n$-adic Lodha--Moore groups have structures of strictly ascending HNN extensions of their subgroups.

\subsection{Case of the Brown--Thompson group}

The Brown--Thompson group $F(n)$ has a well-known HNN decomposition as follows:

\begin{proposition}[{cf.~\cite[Proposition 1.7]{brown1984infinite}}] \label{BT-HNN}
  The Brown--Thompson group $F(n)$ has a strictly HNN decomposition of the form $F(n) \cong F(n)*_{x_0}$ for any $n \geq 2$.
\end{proposition}

\begin{proof}
  Consider the subgroup $F(n)_{\geq 1} \coloneq \langle x_1, x_2, \ldots \rangle$ of $F(n)$.
  This is isomorphic to $F(n)$, and obviously the whole group $F(n)$ is generated by $F(n)_{\geq 1}$ and $x_0$.
  By the relation in (\ref{eq:F(n)-presentation}), we see that $(F(n)_{\geq 1})^{x_0} \subset F(n)_{\geq 1}$ and this inclusion is proper.
\end{proof}


\subsection{Case of the $n$-adic Lodha--Moore groups} \label{LM_HNN}
Zaremsky \cite[Section 2]{zaremsky2016hnn} showed that the Lodha--Moore groups have two types of HNN decompositions, which are called the ``$F$-like'' and ``non-$F$-like'' HNN decomposition.
Recall that Thompson's group $F$ (and $F(n)$) has an HNN decomposition such that the base subgroup is isomorphic to the whole group.
For the  Lodha--Moore groups, ``$F$-like'' means that the base subgroup is larger than the whole group.
On the contrary, ``non-$F$-like'' means that the base subgroup is smaller than the whole group.
In this section, we generalize these decompositions to the $n$-adic Lodha--Moore groups.

\begin{proposition} \label{LM_Flike}
  The $n$-adic Lodha--Moore groups have ``$F$-like'' strictly HNN decompositions of the forms
  \begin{enumerate}
    \item $\G{n} \cong (\yG{n})*_{x_0}$ \label{HNN1}
    \item $\G{n} \cong (\Gy{n})*_{x_0^{-1}}$ \label{HNN2}
    \item $\yG{n} \cong (\yGy{n})*_{x_0^{-1}}$ \label{HNN3}
    \item $\Gy{n} \cong (\yGy{n})*_{x_0}$ \label{HNN4}
  \end{enumerate}
\end{proposition}

\begin{proof}
  First, we prove case (\ref{HNN1}).
  For $\alpha \in \Nn$, we define the subgroup $\G{n}_{\alpha}$ of $\G{n}$ generated by
  \begin{align*}
    \left\{ \x{i}{\alpha \beta} \mid 0 \leq i \leq n-2, \beta \in \Nn \right\} \cup
    \left\{ y_{\alpha \beta} \mid \beta \in \Nn, y_{\alpha \beta} \in Y_0(n) \right\}.
  \end{align*}
  The subgroups $\yG{n}_{\alpha}, \Gy{n}_{\alpha}$ and $\yGy{n}_{\alpha}$ are also defined similarly.
  In particular, $\G{n}_{(n-1)}$ is isomorphic to $\yG{n}$ via $\x{i}{(n-1)\beta} \mapsto \x{i}{\beta}$ and $y_{(n-1)\beta} \mapsto y_{\beta}$.
  Thus it is sufficient to show that $\G{n} \cong (\G{n}_{(n-1)})*_{x_0}$ holds.
  To do this, we need to check that $\G{n}$ is generated by $\G{n}_{(n-1)}$ and $x_0$, and that $(\G{n}_{(n-1)})^{x_0} \subsetneq \G{n}_{(n-1)}$ holds.
  Except for $x_0$, the generators of $\G{n}$ and not contained in $\G{n}_{(n-1)}$ are of the forms
  \begin{enumerate}[label=(\roman*)]
    \item $x_i = \x{i}{\epsilon}$ for $1 \leq i \leq n-2$;
    \item $\x{i}{0\beta}, \ldots, \x{i}{(n-2)\beta}$ for $1 \leq i \leq n-2, \beta \in \Nn$; and
    \item $y_{0\beta}, \ldots, y_{(n-2)\beta}$ for $\beta \in \Nn$, where $y_{0^k}$ is not contained.
  \end{enumerate}

  \underline{Case (i)}\ \ it is clear by the relation $x_0^{-1} x_i x_0 = x_{i+n-1} = \x{i}{n-1} \in \G{n}_{n-1}$.

  \underline{Case (ii)}\ \ for any $1 \leq j \leq n-2$, we have $x_0^{-1} \x{i}{j\beta} x_0 = \x{i}{(n-1)j\beta} \in \G{n}_{(n-1)}$.
  When $j = 0$, suppose that $\beta$ is of the form $0^k l \gamma$ for some $k \geq 0, 1 \leq l \leq n-2$ and $\gamma \in \Nn$.
  Then we have $x_0^{-(k+2)} \x{i}{0\beta} x_0^{k+2} = \x{i}{(n-1)l\gamma} \in \G{n}_{(n-1)}$.
  If $\beta = 0^k$ for some $k \geq 0$, then we have $x_0^{-(k+1)} \x{i}{0\beta} x_0^{k+1} = x_i \x{0}{n-1}^{-1} \in \G{n}_{(n-1)}$.

  \underline{Case (iii)}\ \ this is similar to case (ii).

  Finally, since $x_0^{-1} \x{i}{(n-1)\beta} x_0 = \x{i}{(n-1)(n-1)\beta}$ and $x_0^{-1} y_{(n-1)\beta} x_0 = y_{(n-1)(n-1)\beta}$, the inclusion $(\G{n}_{(n-1)})^{x_0} \subset \G{n}_{(n-1)}$ holds and it is proper.

  By above, case (\ref{HNN4}) immediately follows.
  Also, since $\G{n}_0 \cong \Gy{n}$ and $(\G{n}_{0})^{x_0^{-1}} \subsetneq \G{n}_{0}$, we obtain case (\ref{HNN2}) and (\ref{HNN3}) by the similar argument.
\end{proof}

\begin{proposition} \label{LM_nonFlike}
  The $n$-adic Lodha--Moore groups have ``non-$F$-like'' strictly HNN decompositions of the forms
  \begin{enumerate}
    \setcounter{enumi}{4}
    \item $\yG{n} \cong (\G{n})*_{y_0^{-1}}$ \label{HNN5}
    \item $\Gy{n} \cong (\G{n})*_{y_{n-1}}$ \label{HNN6}
    \item $\yGy{n} \cong (\Gy{n})*_{y_0^{-1}}$ \label{HNN7}
    \item $\yGy{n} \cong (\yG{n})*_{y_{n-1}}$ \label{HNN8}
  \end{enumerate}
\end{proposition}

Before proving this theorem, we show the following lemma:

\begin{lemma} \label{LM_conjugate}
  We have $y_0 x_0 y_0^{-1} \in \G{n}$ and $y_0^{-1} \x{0}{0} y_0 \notin \Gy{n}$.
\end{lemma}

\begin{proof}
  By direct calculation, we have $y_{\alpha} = y_{\alpha(n-1)} y_{\alpha0(n-1)}^{-1} y_{\alpha00} \x{0}{\alpha}$ for all $\alpha \in \Nn$ such that $y_{\alpha} \in Y(n)$.
  Therefore, we obtain
  \begin{align*}
    y_0 x_0 y_0^{-1} & = (y_{0(n-1)} y_{00(n-1)}^{-1} y_{000} \x{0}{0}) x_0 (y_{0(n-1)} y_{00(n-1)}^{-1} y_{000} \x{0}{0})^{-1}                             \\
                     & = y_{0(n-1)} y_{00(n-1)}^{-1} y_{000} (\x{0}{0} x_0 \x{0}{0}^{-1}) y_{000}^{-1} y_{00(n-1)} y_{0(n-1)}^{-1}                          \\
                     & = (y_{0(n-1)} y_{00(n-1)}^{-1} y_{000} x_0^2 \x{0}{n-1}^{-1}) \x{0}{0}^{-1} y_{000}^{-1} y_{00(n-1)} y_{0(n-1)}^{-1}                 \\
                     & = x_0^2 y_{(n-1)(n-1)0} y_{(n-1)0}^{-1} \x{0}{n-1}^{-1} (y_{0} \x{0}{0}^{-1}) y_{000}^{-1} y_{00(n-1)} y_{0(n-1)}^{-1}               \\
                     & = x_0^2 y_{(n-1)(n-1)0} y_{(n-1)0}^{-1} \x{0}{n-1}^{-1} y_{0(n-1)} y_{00(n-1)}^{-1} y_{000} y_{000}^{-1} y_{00(n-1)} y_{0(n-1)}^{-1} \\
                     & = x_0^2 y_{(n-1)(n-1)0} y_{(n-1)0}^{-1} \x{0}{n-1}^{-1}.
  \end{align*}
  By definition, this is in $\G{n}$.
  On the other hand, suppose that $y_0^{-1} \x{0}{0} y_0 = y_{0(n-1)(n-1)}^{-1} y_{0(n-1)0} y_{00}^{-1} y_0$ is in $\Gy{n}$, in particular, $y_{00}^{-1} y_0$ is in $\Gy{n}$.
  Then, it is written as a product of elements in $F(n)$ and of the form $y_{\alpha}$ with $\alpha \neq \epsilon, 0^i$ for any $i \geq 1$.
  Therefore, if we consider the sequence $0^m 000(n-1)00(n-1)(n-1)00(n-1)\cdots \in \NN$ for sufficiently large $m > 0$, then only the first finite number of this sequence must change by $y_{00}^{-1} y_0$.
  However, we have
  \begin{align*}
     & y_{00}^{-1} y_0 (0^m 000(n-1)00(n-1)(n-1)00(n-1)\cdots)         \\
     & = y_0 \circ y_{00}^{-1} (0^m000(n-1)00(n-1)(n-1)00(n-1)\cdots) \\
     & = 0^m0y(0y^{-1}(0(n-1)00(n-1)(n-1)00(n-1)\cdots))              \\
     & = 0^m0y(0000(n-1)(n-1)00(n-1)(n-1)0\cdots)                     \\
     & = 0^m000(n-1)(n-1)(n-1)(n-1)0(n-1)(n-1)(n-1)(n-1)0\cdots.
  \end{align*}
  Hence, this is a contradiction.
\end{proof}

\begin{proof}[Proof of Proposition \ref{LM_nonFlike}]
  First, we prove case (\ref{HNN7}).
  The group $\yGy{n}$ is generated by $\Gy{n} = \langle x_0, x_1, \dots, x_{n-1}, y_{(n-1)0}, y_{(n-1)} \rangle$ and $y_0^{-1}$.
  Thus, we need to check that $(\Gy{n})^{y_0^{-1}} \subset \Gy{n}$ and this inclusion is proper.
  From the relations of $\yGy{n}$, we have $y_0 x_i y_0^{-1} = x_i\ (1 \leq i \leq n-1), y_0 y_{n-1} y_0^{-1} = y_{n-1}$ and $y_0 y_{(n-1)0} y_0^{-1} = y_{(n-1)0}$.
  By Lemma \ref{LM_conjugate}, we have $y_0 x_0 y_0^{-1} \in \Gy{n}$, and $\x{0}{0} \notin (\Gy{n})^{y_0^{-1}}$, and thus the inclusion $(\Gy{n})^{y_0^{-1}} \subset \Gy{n}$ is proper.
  Also, we immediately see case (\ref{HNN5}).

  Similarly to Lemma \ref{LM_conjugate}, we see $y_{n-1}^{-1} x_{n-1} y_{n-1} \in \G{n}$ and $y_{n-1} x_{n-1} y_{n-1}^{-1} \notin \yG{n}$.
  Therefore case (\ref{HNN6}) and (\ref{HNN8}) are also true.
\end{proof}
\subsection{Group cohomology and HNN extensions}
For the arguments in Section \ref{section_notSCY} and \ref{section_homotopy}, we compute the group cohomology with coefficient in the group ring.
Let $G = H*_{\phi, t} = \left\langle H, t \mid t^{-1} a t = \phi(a)\ (a \in A) \right\rangle$ be an HNN extension of a group $H$, where $\phi \colon A \to H$ be a monomorphism from a subgroup $A$ of $H$.
Then we consider the Meyer--Vietoris sequence \cite[Theorem 3.1]{bieri1975mayer}
\begin{align*}
  \cdots \to H^{i}(G;M) \to H^{i}(H;M) \xrightarrow{\alpha} H^{i}(A;M) \to H^{i+1}(G;M) \to \cdots
\end{align*}
with coefficients in any $\bZ[G]$-module $M$.

\begin{lemma} [{\cite[Theorem 0.1]{brown1985cohomology}}] \label{HNN-cohomology}
  Let $H$ and $A$ be of type ${\rm FP}_{n}$.
  Suppose that the map $H^{i}(H;\bZ[H]) \to H^{i}(A;\bZ[H])$ induced by the inclusion $A \hookrightarrow H$ is a monomorphism for some $i \leq n$.
  Then the map
  \begin{align*}
    \alpha \colon H^{i}(H;\bZ[G]) \to H^{i}(A;\bZ[G])
  \end{align*}
  in the Meyer--Vietoris sequence is a monomorphism.
\end{lemma}

The following theorem is already proved in \cite{brown1984infinite, brown1987finiteness}, but we give its proof.

\begin{theorem} [{\cite[Theorem 4.21]{brown1987finiteness}, \cite[Theorem 7.2]{brown1984infinite}}] \label{BT-cohomology}
  For any $n \geq 2$, the Brown--Thompson group $F(n)$ satisfies $H^{i}(F(n);\bZ[F(n)]) = 0$ for all $i \geq 0$.
\end{theorem}

\begin{proof}
  We prove this theorem by induction on $i$.
  Since $F(n)$ is infinite, we have $H^{0}(F(n);\bZ[F(n)]) = 0$.
  Suppose that $H^{i-1}(F(n);\bZ[F(n)]) = 0$.
  By Proposition \ref{BT-HNN}, the Brown--Thompson group $F(n)$ is an ascending HNN extension $(F(n)_{\geq 1})*_{x_0}$.
  In this case, $G = F(n)$ and $A = H = F(n)_{\geq 1}$.
  Then the map $H^{i}(H;\bZ[H]) \to H^{i}(A;\bZ[H])$ is an isomorphism for all $i \geq 0$, in particular monomorphism.
  Also, the group $F(n)$ is of type $\textrm{F}_{\infty}$, in particular type $\textrm{FP}_{\infty}$.
  By Lemma \ref{HNN-cohomology}, in the Meyer--Vietoris sequence
  \begin{align*}
    0 \to H^{i}(F(n);\bZ[F(n)]) \to H^{i}(F(n)_{\geq 1};\bZ[F(n)]) \xrightarrow{\alpha} H^{i}(F(n)_{\geq 1};\bZ[F(n)]) \to \cdots
  \end{align*}
  the map $\alpha$ is a monomorphism.
  Hence $H^{i}(F(n);\bZ[F(n)]) = 0$.
\end{proof}

We obtain a similar result for the $n$-adic Lodha--Moore groups.

\begin{theorem} \label{LM-cohomology}
  For any $n \geq 2$, all of the $n$-adic Lodha--Moore groups $G = \G{n}, \yG{n}, \Gy{n}$ and $\yGy{n}$ satisfy $H^{i}(G;\bZ[G]) = $ for all $i \geq 0$.
\end{theorem}

\begin{proof}
  Similar to Proposition \ref{BT-cohomology}, we prove this by induction on $i$.
  Since $G$ is infinite, $H^{0}(G;\bZ[G]) = 0$ holds\footnote{By Theorem \ref{BT-property} and \ref{LM-property}, the Brown--Thompson group and $n$-adic Lodha--Moore groups are one-ended, and thus we also have $H^{1}(F(n);\bZ[F(n)]) = 0$ and $H^{1}(G;\bZ[G]) = 0$.}.
  Suppose that $H^{i-1}(G;\bZ[G]) = 0$ for any $G = \G{n}, \yG{n}, \Gy{n}$ and $\yGy{n}$.
  By Theorem \ref{appendix_main}, the group $\G{n}$ is of type $\textrm{F}_{\infty}$.
  Moreover, Proposition \ref{LM_Flike} and \ref{LM_nonFlike} state that all of the $n$-adic Lodha--Moore groups are ascending HNN extensions of one of them.
  Therefore, all of them are of type $\textrm{F}_{\infty}$, and thus type $\textrm{FP}_{\infty}$.
  By the same argument as in the proof of Theorem \ref{BT-cohomology}, we obtain $H^{i}(G;\bZ[G]) = 0$.
\end{proof}


\section{Proof of theorems \ref{main_1} and \ref{main_2}} \label{section_notSCY}

First, we observe the (virtual) first Betti numbers for the groups $F(n)$ and $\G{n}$.

\begin{proposition} \label{BT-vBetti}
  For any $n \geq 2$, the Brown--Thompson group $F(n)$ satisfies $b_1(F(n)) = vb_1(F(n)) = n$.
\end{proposition}

\begin{proposition} \label{LM-vBetti}
  For any $n \geq 2$, the $n$-adic Lodha--Moore group $\G{n}$ satisfies $b_1(\G{n}) = vb_1(\G{n}) = n+1$.
\end{proposition}

\begin{proof}[Proof of Propositions \ref{BT-vBetti} and \ref{LM-vBetti}]
  By Theorem \ref{BT-property} and \ref{LM-property}, we have $H_1(F(n); \bZ) \cong \bZ^{n}$ and $H_1(\G{n}; \bZ) \cong \bZ^{n+1}$.
  Let $H$ be a finite index subgroup of $G = F(n)$ or $\G{n}$.
  Then $N \coloneqq \bigcap_{g \in G} g^{-1}Hg$ is a finite index normal subgroup of both $H$ and $G$.
  Furthermore, the center of $G$ is trivial and the commutator subgroup $G'$ is simple.
  Therefore the quotient ${G}/{N}$ is abelian, that is, $N$ contains $G'$.
  On the other hand, the commutator subgroup $N' (\trianglelefteq H')$ is a normal subgroup of $G'$.
  Since $G$ is not virtually abelian, $N$ is not abelian.
  Again, by the simplicity of $G'$, we have $N' = H' = G'$.
  Hence we obtain
  \begin{align*}
    H_1(N; \bZ) = {N}/{N'} = {N}/{G'} \trianglelefteq {G}/{G'} = H_1(G; \bZ).
  \end{align*}
  Similarly, $H_1(H; \bZ) \trianglelefteq H_1(G; \bZ)$ holds.
  As they are finite index normal subgroups, $H_1(N; \bZ)$ and $H_1(H; \bZ)$ are both isomorphic to $H_1(G; \bZ)$.
\end{proof}

From these propositions, we may consider the cases only $F(3), F(4), \G{2} = G_0$ and $\G{3}$, but we prove the facts below for a general case.

Next, We consider the first $L^2$-Betti number of groups.
Roughly speaking, $L^2$-homology is a kind of infinite-dimensional homology theory by using a group von Neumann algebra, and the $L^2$-Betti number is defined as its von Neumann dimension.
For details of $L^2$-homology and $L^2$-invariants, see \cite{luck2001L2invariant} for instance.
The following lemma is useful to check that the first $L^2$-Betti number vanishes.

\begin{lemma} [{\cite[Lemma 2.1]{hillman2002manifold}}]
  Let $G$ be a finitely presented group which is an ascending HNN extension with a finitely generated base.
  Then $b_1^{(2)}(G) = 0$.
\end{lemma}

By this lemma and facts in Section \ref{section_HNN}, we obtain the following results:

\begin{proposition} \label{BT-l2-Betti}
  For any $n \geq 2$, the Brown--Thompson group $F(n)$ satisfies $b_1^{(2)}(F(n)) = 0$.
\end{proposition}

\begin{proposition} \label{LM-l2-Betti}
  For any $n \geq 2$, all of the $n$-adic Lodha--Moore groups $G = \G{n}, \yG{n}, \Gy{n}$ and $\yGy{n}$ satisfy $b_1^{(2)}(G) = 0$.
\end{proposition}

Finally, we need to show that the Hausmann--Weinberger invariants of the Brown--Thompson group and Lodha--Moore groups are positive.
In order to do this, we use the following lemma:

\begin{lemma} [{\cite[Theorem 6]{eckmann1997manifold}}] \label{PD}
  Let $G$ be an infinite, finitely presented, and not virtually-cyclic group with $b_1^{(2)}(G) = 0$ and $H^{2}(G; \bZ[G]) = 0$.
  Let $M$ be a $4$-manifold with $\pi_1(M) \cong G$ and $\chi(M) = 0$.
  Then $M$ is an Eilenberg--Mac Lane space $K(G,1)$ of G, and G is a Poincar\'e duality group of dimension $4$.
\end{lemma}

A group $G$ is a \textit{Poincar\'e duality group} of dimension $n$ for some integer $n > 0$ if it satisfies $H^i(G; \bZ[G]) \cong H_{n-i}(G; \bZ[G])$ for all $i \geq 0$, that is,
\begin{align*}
  H^i(G; \bZ[G]) \cong \left\{
  \begin{array}{ll}
    \bZ & (i = n)    \\
    0   & (i \neq n)
  \end{array}.
  \right.
\end{align*}

\begin{proposition}\label{BT-q-positive}
  For any $n \geq 2$, the Brown--Thompson group $F(n)$ satisfies $q(F(n)) > 0$.
\end{proposition}

\begin{proposition} \label{LM-q-positive}
  For any $n \geq 2$, all of the $n$-adic Lodha--Moore groups $G = \G{n}, \yG{n}, \Gy{n}$ and $\yGy{n}$ satisfy $q(G) > 0$.
\end{proposition}

\begin{proof}[Proof of Propositions \ref{BT-q-positive} and \ref{LM-q-positive}]
  Set $G = F(n), \G{n}, \yG{n}, \Gy{n}$ or $\yGy{n}$.
  By Proposition \ref{BT-l2-Betti} and \ref{LM-l2-Betti}, we have $b_1^{(2)}(G) = 0$.
  Also, since $G$ is infinite, $b_0^{(2)}(G) = 0$ holds.
  Let $M$ be a 4-manifold with $\pi_1(M) \cong G$.
  By a well-known fact for the $L^2$-Betti number, we have
  \begin{align*}
    \chi(M) & = 2b_0^{(2)}(M) - 2b_1^{(2)}(M) + b_2^{(2)}(M)                        \\
            & = 2b_0^{(2)}(G) - 2b_1^{(2)}(G) + b_2^{(2)}(M) = b_2^{(2)}(M) \geq 0.
  \end{align*}
  Therefore $q(G) \geq 0$.
  Suppose that a 4-manifold $M$ above has $\chi(M) = 0$.
  By Theorem \ref{LM-cohomology}, we have $H^{i}(G; \bZ[G]) = 0$ for all $i \geq 0$.
  Thus, $M$ is an Eilenberg--Mac Lane space $K(G, 1)$ and $G$ is a Poincar\'e duality group of dimension 4.
  Namely, $H^{4}(G; \bZ[G]) \cong \bZ$, which is contradiction.
\end{proof}

\begin{remark}
The result \cite[Theorem 0.2]{napier2006thompson} by Napier and Ramachandran implies that the $n$-adic Lodha--Moore group $\G{n}$ is not K\"ahler for all $n \geq 2$.
Indeed, from Theorem \ref{LM-property} (\ref{LM-commutator}) and (\ref{LM-center}), any non-trivial normal subgroup of $\G{n}$ contains the commutator subgroup $\G{n}'$.
Moreover, by Propposition \ref{LM_Flike}, $\G{n}$ is a strictly ascending HNN extension.
\end{remark}

Finally, we give a problem about the Hausmann--Weinberger invariant.
There is an inequality \cite[Theorem 1]{hausmann1985weinberger} of the Hausmann--Weinberger invariant of a group $G$, that is,
\begin{align*}
  2 - 2 b_1(G) + b_2(G) \leq q(G) \leq 2 - 2 \textrm{def}(G),
\end{align*}
where $\textrm{def}(G)$ denotes the deficiency of $G$.
Namely, the maximum of the number of generators minus the number of relations over all finite presentations of $G$.
From this inequality and the fact $\textrm{def}(F) = 0$, Friedl and Vidussi \cite[Theorem]{friedl2015thompson} showed that $0 < q(F) \leq 2$.
However, this invariant has not yet been computed completely.
Note that $G_0$ has a well-known presentation with three generators and nine relations \cite{lodha2016nonamenable}. 
Hence the Hausmann--Weinberger invariant for the Lodha--Moore group is even more unclear. 
\begin{problem}
Determine the Hausmann--Weinberger invariant for the Brown--Thompson group and the ($n$-adic) Lodha--Moore groups.
\end{problem}


\section{Homotopy groups at infinity} \label{section_homotopy}

In this section, we prove the triviality of homotopy groups at infinity for $n$-adic Lodha--Moore groups, which is the fourth condition of Geoghegan's conjecture.

Let $G$ be a group of type $\textrm{F}_n$ for some $n \geq 0$.
Let $X$ be an Eilenberg--Mac Lane space $K(G,1)$ with finite $n$-skeleton, and $\widetilde{X}$ its universal covering.
A group $G$ is \textit{$(n-1)$-connected at infinity} if for any compact subset $C \subset \widetilde{X}$, there exists a compact subset $C \subset D \subset \widetilde{X}$ such that any map $S^{k} \to \widetilde{X} \setminus D$ extends to a map $B^{k+1} \to \widetilde{X} \setminus C$ for all $k < n$.
In other words, the natural map $\pi_k(\widetilde{X} \setminus D) \to \pi_k(\widetilde{X} \setminus C)$ induced by the inclusion is trivial.

Note that $G$ is $0$-connected at infinity if and only if it has one-ended.
If a group $G$ is 1-connected at infinity, then it is also called \textit{simply-connected at infinity}.
If $G$ is $n$-connected at infinity for all $n$, then we say that it has \textit{trivial homotopy groups at infinity}.

\begin{lemma} [{\cite[Theorem 3.1]{mihalik1985end}}]
  Let $G$ be a one-ended finitely presented group and $\phi \colon G \to G$ a monomorphism.
  Then its HNN extension $G*_{\phi, t}$ is simply connected at infinity.
\end{lemma}

From this lemma, Theorem \ref{LM-property} (2), Proposition \ref{LM_Flike} and \ref{LM_nonFlike}, we have the following:

\begin{theorem}
  All of the $n$-adic Lodha--Moore groups are simply connected at infinity.
\end{theorem}

In order to prove the triviality of homotopy groups at infinity, we use the following facts:

\begin{lemma} [{\cite[Theorem 17.2.1]{geoghegan2008group}}]
  If a group $G$ is of type ${\rm F}_n$ and simply-connected at infinity, then $G$ is $(n-1)$-connected at infinity if and only if $G$ is $(n-1)$-acyclic at infinity with respect to $\bZ$.
\end{lemma}

This lemma can be called the ``Hurewicz theorem at infinity''.
We do not explain the definition of the $n$-acyclicity at infinity, since we only need its sufficient condition; if $H^k(G; \bZ[G]) = 0$ for all $k \leq n$, then $G$ is $(n-1)$-acyclic at infinity with respect to $\bZ$ \cite[Theorem 13.3.3 (ii)]{geoghegan2008group}\footnote{More precisely, a group $G$ is $(n-1)$-acyclic at infinity with respect to a commutative ring $R$ if and only if $H^k(G; R[G]) = 0$ for all $k \leq n$ and $H^{n+1}(G; R[G])$ is a torsion-free $R$-module.}.
From Theorem \ref{LM-cohomology}, we see that all of the $n$-adic Lodha--Moore groups are $n$-acyclic at infinity for all $n \geq 0$.
Therefore we obtain the following theorem:

\begin{theorem}
  All of the $n$-adic Lodha--Moore groups have trivial homotopy groups at infinity.
\end{theorem}

Recall that for any $n \geq 2$, the $n$-adic Lodha--Moore group admits no nontrivial direct and free product decompositions \cite{kodama2023n}.
Therefore, we get ``non-trivial'' infinitely many groups that satisfy all of Geoghegan's conditions.


\section*{Acknowledgements}
The first and second authors are partially supported by JSPS KAKENHI Grant number 24K22836 and 24KJ0144, respectively.
We would like to thank Professor Daniel Farley for the comments on the finiteness property of the $n$-adic Lodha--Moore group. 
\appendix
\section{Finiteness property of the group $\G{n}$} \label{appendix_Finf}
Fix $n \geq 2$.
In this appendix, we give a sketch of the proof of the following:
\begin{theorem} \label{appendix_main}
  The group $\G{n}$ is of type ${\rm F}_\infty$.
\end{theorem}
By \cite[Proposition 1.1]{brown1987finiteness}, it is sufficient to show the following:
\begin{theorem} \label{thm_G_0(n)_F_infty}
  There is a connected CW complex $X$ such that $\G{n}$ acts by cell permuting homeomorphisms such that the following hold:
  \begin{enumerate}
    \item $X$ is contractible;
    \item The quotient $X/G$ has finitely many cells in each dimension; and
    \item The stabilizer of each cell has type {$\rm F_\infty$}.
  \end{enumerate}
\end{theorem}
The full proof is obtained by natural generalizations of \cite{lodha2020nonamenable}.
All undefined notions of $\G{n}$ are defined in \cite{kodama2023n}.
\subsection{The cells of the complex $X$}
\begin{definition}
  The $0$-skeleton $X^{(0)}$ is defined to be the set of right cosets of $F(n)$ in $\G{n}$.
  Two cosets $F(n)\tau_1$ and $F(n)\tau_2$ in $X^{(0)}$ are connected by an edge in the $1$-skeleton $X^{(1)}$ if the double coset $F(n)\tau_1\tau_2^{-1} F(n)$ is equal to one of the following double cosets:
  \begin{align*}
     & F(n)y_{(n-1)0}F(n)                      &  & F(n)y_{(n-1)0}^{-1}F(n)                 \\
     & F(n)y_{(n-1)00}^{-1}y_{(n-1)0(n-1)}F(n) &  & F(n)y_{(n-1)00}y_{(n-1)0(n-1)}^{-1}F(n)
  \end{align*}
\end{definition}
We always consider the natural action of $\G{n}$ on $X^{(0)}$.
To define the higher dimensional skeletons, we first recall some notions.
\begin{definition}
  A word $\y{s_1}{t_1}\cdots \y{s_m}{t_m}$ on $Y_0(n)$ is \textit{special} if there exists an $n$-ary tree $T$ such that $s_1, \dots, s_m$ are all leaves of $T$, each pair of leaves $s_i, s_{i+1}$ have $n-2$ leaves between them, and $t_1, \dots, t_m \in \{1, -1\}$ with $t_i=-t_{i+1}$.
\end{definition}
Note that for $\tau_1, \tau_2 \in \G{n}$, the following are equivalent:
\begin{itemize}
  \item $F(n)\tau_1$ and $F(n)\tau_2$ are connected in $X^{(1)}$.
  \item there exists a standard form $f\y{s_1}{t_1}\cdots \y{s_m}{t_m}$ of $\tau_1 \tau_2^{-1}$ such that $\y{s_1}{t_1}\cdots \y{s_m}{t_m}$ is special.
\end{itemize}

Two special words $\y{s_1}{t_1}\cdots\y{s_m}{t_m}$ and $\y{u_1}{v_1} \cdots \y{u_l}{v_l}$ on $Y_0(n)$ are \textit{independent} if for any pair $\y{s_j}{t_j}$ and $\y{u_k}{v_k}$, we have $\y{s_j}{t_j}\y{u_k}{v_k}=\y{u_k}{v_k}\y{s_j}{t_j}$.
A list of special words is \textit{independent} if special forms are pairwise independent.
A list of special words $\tau_1, \dots, \tau_m$ is \textit{sorted} if it is independent and if $i<j$, then any elements $\y{s}{t}$ in $\tau_i$ and $\y{u}{v}$ in $\tau_j$ satisfy $s<u$, where the order $<$ is the usual lexicographical order.

For a sorted list $\tau_1, \dots, \tau_m$ and $X=\{j_1, \dots, j_l\} \subseteq \{1, \dots, m\}$ with $j_1<\cdots <j_l$, we define $\tau_X \coloneqq \tau_{j_1} \cdots \tau_{j_l}$.
To define higher dimensional cells, we first define $\cH$ as the set of finite sets of edges in $X^{(1)}$.
\begin{definition}
  An element $C$ of $\cH$ is a subgraph of $X^{(1)}$ which is determined by the following:
  \begin{enumerate}
    \item A sorted list of special words $\tau_1, \dots, \tau_m$.
    \item An element $\tau \in \G{n}$.
  \end{enumerate}
  Then $C$ is the induced subgraph consisting of the vertex set
  \begin{align*}
    \{F(n)\tau_X \tau \mid X \subseteq \{1, \dots, m\}\}
  \end{align*}
  in $X^{(1)}$.
\end{definition}
Note that any element in $\cH$ are not represented by a unique sorted list of special words and an element in $\G{n}$.
Especially, elements in $\G{n}$ required in (2) can be taken as words on $Y_0(n)$.
Observe that any $1$-cell $(F(n)\tau_1, F(n)\tau_2)$ is in $\cH$.
\subsection{The complex $X$}
Let $m \geq 1$.
We first recall the notion of hyperplane arrangement in $\bR^m$ and then define the cluster complex.
See \cite[Section 3.4 and 5]{lodha2020nonamenable} for details.

Let $\bR^m$ be the Euclidean space with the usual orthogonal basis.
Let $x_1, \dots, x_m$ be coordinates with this basis.
For $a_1, \dots, a_{m+1} \in \bR$, an \textit{affine hyperplane} is an $(m-1)$-dimensional subspace
\begin{align*}
  \{(x_1, \dots, x_m) \mid a_1x_1+\cdots+a_mx_m=a_{m+1}\}.
\end{align*}
For simplicity, we write this set as $\{a_1x_1+\cdots+a_mx_m=a_{m+1}\}$.

A finite \textit{hyperplane arrangement} $\cA$ is a finite set of affine hyperplanes in $\bR^m$.
We write $\cR(\cA)$ for the set of connected components of $\bR^m \setminus \bigcup_{H \in \cA}H$.
For affine hyperplanes in $\cA$, we call the subspace obtained from taking intersection \textit{flat}.
The trivial intersection, i.e., $\bR^m$, is also called frat.
For a flat $T$, we define the hyperplane arrangement $\cA \restriction T$ as
\begin{align*}
  \cA \restriction T \coloneqq \{T \cap H \mid H \in \cA, T \not \subseteq H\}.
\end{align*}
Let $\cR(\cA) \restriction T$ be the set of connected components of $T \setminus \bigcup_{H \in A \restriction T} H$.
Then
\begin{align*}
  \bigcup_{\text{$T$ is a flat of $\cA$}} \cR(\cA) \restriction T
\end{align*}
is called the \textit{face complex of the arrangement $\cA$}.
The face complex provides a cellular structure on $\bR^m \cup \partial \bR^m=\mathbb{B}^m$.
\begin{example} \label{example_face_complex}
  Let $\cA=\{ \{x_1=0\}, \{x_2=0\}, \{x_1=x_2\}\}$ be a hyperplane arrangement of $\bR^2$.
  Then the set of flats of $\cA$ is $\{\bR^2, \{x_1=0\}, \{x_2=0\}, \{x_1=x_2\}, \{(0, 0)\}\}$.
  Then the face complex of $\cA$ is given as in Figure \ref{Fig_flat_complex}.
  \begin{figure}[tbp]
    \begin{center}
      \includegraphics[width=0.3\linewidth]{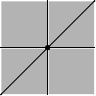}
    \end{center}
    \caption{An example of a flat complex (before the one-point compactification). }
    \label{Fig_flat_complex}
  \end{figure}
\end{example}
To define our complex, we consider two types of affine hyperplanes in $\bR^m$:
\begin{enumerate}
  \item $\{x_i=0\}$ or $\{x_i=1\}$ for some $1 \leq i \leq m$.
  \item $\{x_i=x_{i+1}\}$ for some $1 \leq i \leq m-1$.
\end{enumerate}
The first hyperplanes are called type 1, and the second type 2.
Then a hyperplane arrangement $\cA$ is \textit{admissible} if $\cA$ consists of hyperplanes of type 1 and 2, and contains all hyperplanes of type 1.
\begin{definition}
  An $m$-cluster is a CW subdivision of the $m$-dimensional cube $[0, 1]^m$, which is a restriction of the face complex of an admissible hyperplane arrangement of $\bR^m$ to $[0, 1]^m$.
  For an admissible hyperplane arrangement $\cA$ of $\bR^m$, we write $\cC(\cA)$ as the restriction of the face complex of $\cA$ to $[0, 1]^m$, which is a subcomplex.
  The subcomplex $\cC(\cA)$ is called the $m$-cluster (or simply cluster) associated with $\cA$.
\end{definition}
\begin{example}
  For the admissible hyperplane $\cA$ defined in Example \ref{example_face_complex}, the subcomplex $\cC(\cA)$ is illustrated in Figure \ref{Fig_m-cluster}.
  \begin{figure}[tbp]
    \begin{center}
      \includegraphics[width=0.3\linewidth]{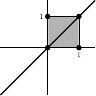}
    \end{center}
    \caption{An example of $\cC(\cA)$. }
    \label{Fig_m-cluster}
  \end{figure}
\end{example}
For a cluster $\cC(\cA)$, a \textit{subcluster} is a subcomplex obtained by an intersection of $\cC(\cA)$ with a frat in $\cA$.
Then we define a \textit{complex of clusters} as a CW complex obtained by gluing clusters along subclusters using cellular homeomorphisms.

\begin{definition}
  Two special words $\tau_1=\y{s_1}{t_1}\cdots \y{s_m}{t_m}$ and $\tau_2=\y{u_1}{v_1} \cdots \y{u_l}{v_l}$ are \textit{alternating} if $\tau_1 \tau_2$ is a special word.
  If $\tau_1, \tau_2$ or $\tau_1, \tau_2^{-1}$ are alternating, $\tau_1, \tau_2$ are called \textit{consecutive}.
  Let $\tau_1, \dots, \tau_m$ be a sorted list of special words.
  We say that this list is \textit{proper} if for any $i \in \{1, \dots, m-1\}$, whenever $\tau_i, \tau_{i+1}$ is consecutive, then it is also alternating.
\end{definition}
\begin{lemma}[{cf.~\cite[Lemma 7.20]{lodha2020nonamenable}}]
  Any element in $\cH$ is represented by a proper sorted list of special words and an element in $\G{n}$.
\end{lemma}
Now we identify 1-skeletons of clusters with elements in $\cH$.
\begin{proposition}[{cf.~\cite[Proposition 7.23]{lodha2020nonamenable}}]
Let $C \in \cH$.
Take a proper sorted list of special words $\tau_1, \dots, \tau_m$  and an element in $\G{n}$ representing $C$.
Then $C$ is isomorphic to the $1$-skeleton of an $m$-cluster.
\begin{sproof}
For any $\tau \in \G{n}$, we have that $C \cdot \tau$ is also in $\cH$ (cf.~\cite[Lemma 7.7]{lodha2020nonamenable}).
Thus by acting on $C$ with an appropriate element, it is enough to consider the case that the element in $\G{n}$ representing $C$ is the identity element.
We define
\begin{align*}
  \textbf{Y}\coloneq\{i \in \{1, \dots, m-1\} \mid \text{$\tau_i \tau_{i+1}$ is a special form}\}.
\end{align*}
and the hyperplane arrangement $\cA$ in $\mathbb{R}^m$ consisting of hyperplanes $\{ x_i=0 \}, \{ x_i=1 \}$ for $1 \leq i \leq m$ and $\{ x_i=x_{i+1} \}$ for $i \in \textbf{Y}$.
Then $\cA$ is admissible, and thus we have the $m$-cluster $\cC(\cA)$.

We construct a graph isomorphism between $\cC(\cA)^{(1)}$ and $C$ via identifying $\cC(\cA)^{(1)}$ with the graph defined as follows:
a vertex is a subset of $\{1, \dots, m\}$.
An edge is a partition of $\{1, \dots, m\}$ into three sets $X_1, X_2, X_3$ such that $X_2$ is of the form $\{j, j+1, \dots, k\} \subset \{1, \dots, m\}$ with $i \in \textbf{Y}$ when $j \leq i \leq k-1$.
This edge connects the vertices $X_3$ and $X_2\cup X_3$.
We identify the vertex $X \subseteq \{1, \dots, m\}$ with the $0$-cell of $\cC(\cA)$ of coordinates $(y_1, \dots, y_m)$ where $y_i=0$ if $i \not\in X$ and $y_i=1$ if $i \in X$, and the edge $(X_1, X_2, X_3)$ with the $1$-cell of $\cC(\cA)$ connecting the $0$-cells $(y_1, \dots, y_m)$ and $(z_1, \dots, z_m)$ given by the following coordinates:
\begin{align*}
   & (y_1, \dots, y_m) & \text{$y_i=0$ if $i \in X_1 \cup X_2$ and $y_i=1$ if $i \in X_3$} \\
   & (z_1, \dots, z_m) & \text{$z_i=0$ if $i \in X_1$ and $z_i=1$ if $i \in X_2 \cup X_3$}
\end{align*}
This is precisely the intersection of the following collection of hyperplanes in $[0, 1]^m$:
\begin{align*}
  \{ \{x_i=0\} \mid i \in X_1\} \cup \{\{x_i =1\}\mid i\in X_3\} \cup \{\{x_i=x_j\}\mid i,j \in X_2\}.
\end{align*}
A graph isomorphism between $\cC(\cA)^{(1)}$ and $C$ is defined as follows:
the 0-cell of $\cC(\cA)$ given by $X \subseteq \{1, \dots, m\}$ corresponds to the 0-cell $F(n)\tau_X$ of $C$, and the 1-cell of $\cC(\cA)$ given by a triple $(X_1, X_2, X_3)$ (satisfying the above) corresponds to the 1-cell connecting the pair $F(n)\tau_{X_3}$ and $F(n)\tau_{X_2 \cup X_3}$.
\end{sproof}
\end{proposition}

Note that this identification does not depend on the choice of the sorted list of special words and the element in $\G{n}$.

Since $\cC(\cA)$ has a complex structure, by using the above identification, we can define higher cells of $X$.
\begin{definition}
  Let $C \in \cH$ and let $\cC(\cA)$ be the cluster such that the $1$-skeleton of $\cC(\cA)$ is identified with $C$.
  A \textit{filling} $\textbf{C}$ of $C$ is given by attaching higher cells of $\cC(\cA)$ to $C$ in $X^{(1)}$.
  The complex $X$ is defined as $X=\bigcup_{C\in\cH}\textbf{C}$.
\end{definition}
The action of $\G{n}$ on $X^{(1)}$ is naturally extended to $X$.
To show that this space $X$ is a CW complex, we need to prove the following two propositions.
We omit their proofs for the sake of length, but they are all obtained by a natural generalization of the argument in \cite{lodha2020nonamenable}.
\begin{proposition}[{cf.~\cite[Proposition 7.8]{lodha2020nonamenable}}]
  Let $C_1, C_2 \in \cH$.
  If $C_1 \cap C_2$ is nonempty, then it is in $\cH$.
\end{proposition}
\begin{proposition}[{cf.~\cite[Proposition 7.27]{lodha2020nonamenable}}]
  Let $C_1, C_2 \in \cH$ with $C_1 \subset C_2$.
  Then $\textbf{C}_1$ is the subcomplex of $\textbf{C}_2$ obtained from $C_1 \subset \textbf{C}_2$.
\end{proposition}
\subsection{Ideas for the proof of Theorem \ref{thm_G_0(n)_F_infty}}
\subsubsection{$X$ is contractible}
It is sufficient to show that for any $m \geq 1$, the $m$-th homotopy group $\pi_m(X)$ is trivial.
If $m=1$, we can directly show that any closed curve and the trivial curve are homotopic.
Roughly speaking, this is due to the following facts:
\begin{enumerate}
  \item Any closed path in $X^{(1)}$ whose endpoint is $F(n)$ is homotopic to the path $F(n), F(n)(\y{s_k}{t_k}), \dots, F(n)(\y{s_k}{t_k} \cdots \y{s_1}{t_1})=F(n)$.
  \item for any element $g$ in $\G{n}$, there is a unique normal form representing $g$.
  \item Transformations from words to the normal forms can be realized as homotopies on $X^{(1)}$.
\end{enumerate}
Note that to complete these discussions, it is necessary to define certain cells of $X$ and generalize the arguments in \cite[Section 4]{lodha2020nonamenable} (or \cite{lodha2016nonamenable}), and this generalization of the latter has already been given in \cite{kodama2023n}.

For $m\geq 2$, the discussion becomes more complicated. As in \cite{lodha2020nonamenable}, the key idea is to ``sandwich'' nonpositively curved cube complexes (in the sense of Gromov) between $X$ and any image of $S^m$ to $X$ under a continuous map.
Recall that nonpositively curved cube complexes are aspherical \cite{gromov1987hyperbolic,bridson2013metric}.
Then the following holds:
\begin{proposition}[{cf.~\cite[Section 10]{lodha2020nonamenable}}]
  For any finite subcomplex $Y$ of $X$, there exist a nonpositively curved cube complex $Z$ and two continuous maps $f\colon Y \to Z$, $g\colon Z \to X$ such that the following diagram commutes:
  \begin{equation*}
    \begin{tikzcd}
      & Z \arrow[rd, "g"] &   \\
      Y \arrow[ru, "f"] \arrow[rr, hook] &                   & X
    \end{tikzcd}
  \end{equation*}
\end{proposition}
\subsubsection{The quotient $X/G$ has finitely many cells in each dimension}
In \cite[Section 8.2]{lodha2020nonamenable}, for a given dimension $m$, a sufficient condition is given for two $m$-clusters to be in the same orbit.
By this condition, it follows that the number of orbits is finite when $n=2$.
These arguments are easily generalized to the $n$-adic case.
\subsubsection{The stabilizer of each cell has type {$\rm F_\infty$}.}
The main idea is to reduce the group to the intersection of stabilizer subgroups (of $F(n)$) of $0$-cells.
Then for each $0$-cell $x=F(n)\y{s_1}{t_1}\cdots \y{s_m}{t_m}$ and for any $f \in \stab{F(n)}{x}$, we have $f=f_1 f_2 f_3$ where $\supp{f_1} \subset (\ov{0}, s_1\ov{0})$, $\supp{f_2} \subset (s_1\ov{0}, s_m\ov{(n-1)})$, and $\supp{f_3} \subset (s_m\ov{(n-1)}, \ov{(n-1)})$.
This give a map $\stab{F(n)}{x} \to F(n) \times F_{3, m} \times F(n); f \to (f_1, f_2, f_3)$.
Since the map $y$ are almost the same for $n=2$ as for the general case, we obtain a similar result.
See \cite[Section 6 and 8]{lodha2020nonamenable} for details.
\bibliographystyle{plain}
\bibliography{references}
\bigskip

Yuya KODAMA

\address{
  Graduate School of Science and Engineering
  Kagoshima University
  1-21-35 Korimoto, Kagoshima
  Kagoshima 890-0065, Japan
}

\textit{E-mail address}: \href{mailto:yuya@sci.kagoshima-u.ac.jp}{\texttt{yuya@sci.kagoshima-u.ac.jp}}

\bigskip

Akihiro TAKANO

\address{Department of Mathematics, Graduate School of Science
  Osaka University
  1-1 Machikaneyama, Toyonaka
  Osaka 560-0043, Japan
}

\textit{E-mail address}: \href{mailto:takano.akihiro.sci@osaka-u.ac.jp}{\texttt{takano.akihiro.sci@osaka-u.ac.jp}}
\end{document}